\documentclass[11pt]{amsart}
\usepackage{amsmath, amscd, amssymb}
\usepackage{latexsym, amssymb, amsmath, amscd, amsfonts}
\usepackage[mathscr]{eucal}
\usepackage{mathrsfs}
\usepackage{color}
\usepackage{concrete,euler,textcomp}
\usepackage{youngtab}
\usepackage[T1]{fontenc}

\newtheorem{theorem}{Theorem}[subsection]

\newtheorem{acknowledgement*}{Acknowledgement}

\newtheorem{corollary}[theorem]{Corollary}

\newtheorem{definition}[theorem]{Definition}

\newtheorem{lemma}[theorem]{Lemma}

\newtheorem{proposition}[theorem]{Proposition}
\newtheorem{remark}[theorem]{Remark}

\input xy
\xyoption{all}

\newcommand{\A}{\mathcal{A}}

\newcommand{\HH}{\mathcal{G}}
\newcommand{\I}{\mathcal{I}}
\newcommand{\F}{\mathcal{F}}

\title[A quantum analogue of Kostant's theorem]{A quantum analogue of Kostant's theorem for the general
linear group}

\begin{document}

\author{Avraham Aizenbud}
\email{aizenr@yahoo.com}
\address{Department of Mathematics, Weizmann Institute of Science,
Ziskind Building,
Rehovot 76100, ISRAEL}
\author{Oded Yacobi}
\email{yacobi@gmail.com}
\address{School of Mathematical Sciences, Tel Aviv University, Tel Aviv
69978, Israel} 

\maketitle

\begin{abstract}
A fundamental result in representation theory is Kostant's theorem which
describes the algebra of polynomials on a reductive Lie algebra as a module
over its invariants.  We prove a quantum analogue of this theorem
for the general linear group, and from this deduce the analogous result for reflection equation algebras.
\end{abstract}

\section{Introduction}
A classical theorem of Kostant's states that the algebra of polynomials $\mathcal{O}(\mathfrak{g})$ on a reductive Lie algebra $\mathfrak{g}$ is free as a module over the invariant polynomials
$\mathcal{O}(\mathfrak{g})^{G}$ (see \cite{K}).  This result, which was later generalized by Kostant and Rallis to arbitrary symmetric pairs (see \cite{KR}), is fundamental to
representation theory.  In particular, it plays an important role in understanding geometric properties of the nilpotent cone, and representation theoretic properties of its ring of regular functions (see e.g. Chapter 6 of \cite{CG}).

  In this paper we prove a quantum analog of Kostant's Theorem for the general linear group.  Namely, we show that the coordinate ring of quantum matrices $\A=\mathcal{O}(M_{q}(n))$ is free as a module over $\I$, the subalgebra of invariants under the adjoint coaction of $\mathcal{O}(GL_{q}(n))$, for $q$ not a root of unity or $q=1$.  At $q=1$ this is a restatement of Kostant's Theorem for the general linear group.

Several proofs of Kostant's Theorem in the classical case have appeared over the last forty years.
Our proof in the quantum case is adapted from the argument in \cite{BL}, which is similar to an earlier proof appearing in \cite{W}.
In order to explain our approach, we briefly sketch their argument in the case of the general linear group.

Consider the filtration on $\mathcal{O}(\mathfrak{gl}_{n}(\mathbb{C}))$ given by $deg(x_{ij})=\delta_{ij}$, where $\{ x_{ij} \}$ are
the standard coordinates on $\mathfrak{gl}_{n}(\mathbb{C})$.  Now let $I$ be the subalgebra of $\mathcal{O}(\mathfrak{gl}_{n}(\mathbb{C}))$
consisting of $GL_{n}(\mathbb{C})$-invariant polynomials, with the induced filtration.  Then the fact that $\mathcal{O}(\mathfrak{gl}_{n}(\mathbb{C}))$
is free over $I$, follows from the fact that $gr(\mathcal{O}(\mathfrak{gl}_{n}(\mathbb{C}))$ is free over $grI$.  This, in turn, follows from the
standard fact that the algebra of polynomials is free as a module over the ring of symmetric polynomials.

While our proof is based on the same idea, the quantum setting presents new complications. Indeed, the above filtration cannot be
adapted to the quantum setting in a manner that is compatible with the algebra structure.  Therefore we have to use a more subtle approach,
whereby we use a succession of filtrations each of which slightly simplifies the relations.

More precisely, we first construct a filtration on $\A$ that is compatible with the algebra structure.  We then consider the associated
graded algebra $\A'=gr\A$ as a module over $\I'=gr\I$.  However, our filtration is weak in the sense that the freeness of $\A'$ over $\I'$
does not follow from standard facts.  Nevertheless, the algebra $\A'$ has slightly simpler relations than the original algebra.  This allows
us to define a ``stronger'' filtration on $\A'$ and again consider it's associated graded algebra $\A''$.

We continue in such a way until we
can reduce to the standard fact mentioned above.  This argument relies on a theorem of Domokos and Lenagan (\cite{DL}) which gives an explicit
presentation of $\I$.  It is conjectured that the result of Domokos and Lenagan extends to arbitrary $q$, in which case our result will extend as well.

Our result implies the analogous statement in the setting of reflection equation algebras (also known as ``braided matrices'').  The reflection equation algebra $\mathcal{S}$ is another quantization of the coordinate ring of $n \times n$ matrices, which is also endowed with an adjoint coaction of the quantum general linear group.  In contrast to $\mathcal{A}$, the reflection equation algebra $\mathcal{S}$ is a comodule-algebra, and moreover its invariants with respect to the adjoint coaction are central.  We prove that $\mathcal{S}$ is free as a left $\mathcal{J}$-module, where $\mathcal{J}$ is the algebra of invariants.

As a corollary of our main result, we obtain a (non-canonical) equivariant decomposition of $\A$ as a tensor product of $\I$ and a $\mathcal{G}$-comodule $\mathcal{H}$.   We also obtain the analogous result for $\mathcal{S}$; the benefit of this formulation is that the $\mathcal{G}$-comodule corresponding to $\mathcal{H}$ is now an algebra.  This algebra can be regarded as
a quantization of the algebra of functions on the nilpotent cone (see \cite{D}).  It would be interesting to make these decompositions
canonical by defining a quantum analogue of the harmonic polynomials.

In the literature there are other quantum analogues of Kostant's
Theorem.  In \cite{JL} it is proven that the locally finite part of
the quantum enveloping algebra $U_{q}(\mathfrak{g})$ of a semisimple
Lie algebra $\mathfrak{g}$ is free over its center.  Another
analogue appears in \cite{B}, where it is shown that the algebra
$\mathcal{O}_{q}(G)$ is free over its invariants with respect to its
adjoint coaction for simple simply connected $G$ for generic $q$.
From our result one can deduce this theorem for the general linear
group.  However, it seems difficult to show the reverse implication.
%
%
\subsection*{Acknowledgements}
We are grateful to Matyas Domokos for helpful comments on a preliminary version of this paper, and in particular for pointing out to us the connection with reflection equation algebras.  We also thank Joseph Bernstein, Anthony Joseph, and Nolan Wallach for helpful conversations.

%
%

\section{Preliminaries}
\subsection{Filtered algebras}
We begin by recording some standard notations regarding filtrations.
Let $(V,F)$ be a linear space $V$ with a filtration $F$.
All filtrations we consider in this paper will be ascending. For $x
\in V$ we denote by $deg_{F}(x)=min\{d:x \in F^{d}V \}$.  The symbol
map $\sigma_{F}^{d}:F^{d}V \rightarrow gr_{F}^{d}V$ maps an element
$x$ to $x + F^{d-1}V$.  For any $x \in V$ we let
$\sigma_{F}(x)=\sigma_{F}^{deg_{F}(x)}(x)$.
\begin{lemma}
\label{symbprod} Let $(A,F)$ be a filtered algebra and let $x,y \in
A$.  If $deg_{F}(xy)=deg_{F}(x)+deg_{F}(y)$ then
$\sigma_{F}(xy)=\sigma_{F}(x)\sigma_{F}(y)$.
\end{lemma}
\begin{lemma}[Lemma 4.2, \cite{BL}]
\label{BL} Let $(M,F)$ be a filtered module over a filtered algebra
$(A,F)$ and $\{m_{k}\}$ a family of elements of $M$. Suppose that
the symbols $\sigma_{F}(m_{k})$ form a free basis of the
$gr_{F}(A)$-module $gr_{F}(M)$. Then $\{m_{k}\}$ is a free basis of
the $A$-module $M$.
\end{lemma}

\begin{lemma}
\label{lemmaCG}
Let $A=\bigoplus_{d \geq 0}A_{d}$ be a unital graded associative algebra, and let $I=\bigoplus_{d \geq 0}I_{d} \subset A$ be a unital graded subalgebra.  Suppose $I_{0}=A_{0}=span\{1\}$, and that $A$ is a free left $I$-module.  Define $I_{+}= \bigoplus_{d > 0}I_{d}$ and let $H \subset A$ be a graded linear complement to $A I_{+}$:
$$
H \oplus A I_{+}=A.
$$
Then the multiplication map
$$
H \otimes I \rightarrow A
$$
is an isomorphism of left $I$-modules.
\end{lemma}

For a proof of this lemma see the proof
of Theorem 6.3.3 in \cite{CG} (p. 319).

\subsection{Quantum groups}
We recall the definition of the quantum $n \times n$ matrices and
the quantum general linear group.  Fix $q \in \mathbb{C}^{\times}$ and let $\mathcal{O}(M_{q}(n))$ be the
bi-algebra of quantum $n \times n$ matrices, i.e. $\mathcal{O}(M_{q}(n))$
is the $\mathbb{C}$-algebra generated by indeterminates $ x_{ij} $
$( i,j=1,...,n )$ subject to the following relations:
\begin{eqnarray}
 x_{ij}x_{il} &=& qx_{il}x_{ij}       \\
 x_{ij}x_{kj} &=& qx_{kj}x_{ij}  \\
 x_{il}x_{kj} &=& x_{kj}x_{il}  \\
 x_{ij}x_{kl}-x_{kl}x_{ij} &=& (q-q^{-1})x_{il}x_{kj}
\end{eqnarray}
where $1 \leq i < k \leq n$ and $1 \leq j < l \leq n$.

We introduce a diagrammatic shorthand to work with these relations.
First consider the case $n=2$.  The relations defining
$\mathcal{O}(M_{q}(2))$ are encapsulated in the following diagram:
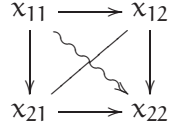
\begin{equation*}
\xymatrix{ {x_{11}} \ar[d]\ar@{~>}[dr]\ar[r]
& {x_{12}} \ar[d] \ar@{-}[ld] \\
{x_{21}} \ar[r] & {x_{22}} }
\end{equation*}
Here, if there is an undirected edge between $x_{ij}$ and $x_{kl}$ then
$[x_{ij},x_{kl}]=0$.  A directed edge from $x_{ij}$ to $x_{kl}$
means $x_{ij}x_{kl}= qx_{kl}x_{ij}$.  Finally, a curly directed edge
from $x_{ij}$ to $x_{kl}$ means they satisfy the ``complicated''
relation (4) above.

In general, the relations defining $\mathcal{O}(M_{q}(n))$ can be
expressed as: every $2 \times 2$ submatrix of
$\mathcal{O}(M_{q}(n))$ generates a copy of $\mathcal{O}(M_{q}(2))$.
For instance, if $n=3$ then the submatrix obtained by choosing the
first and third row and the second and third columns contributes the
following relations:
\begin{equation*}
\xymatrix{ {\bullet} & {\bullet} \ar@/_/[dd]\ar@{~>}[ddr]\ar[r]
& {\bullet} \ar@{-}[ddl] \ar@/^/[dd] \\
{\bullet} & {\bullet} & {\bullet} \\
{\bullet} & {\bullet} \ar[r] & {\bullet} }
\end{equation*}

\begin{theorem}[Theorem 3.5.1, \cite{PW}]
\label{PBW} $\mathcal{O}(M_{q}(n))$ has a PBW-type basis consisting
of monomials $\{ x^{a} : a \in M_{n}(\mathbb{N}) \}$, where
$x^{a}=x_{11}^{a_{11}}x_{12}^{a_{12}} \cdots x_{nn}^{a_{nn}}$.
\end{theorem}
Fix $n \in \mathbb{N}$ and let $\mathcal{A}=\mathcal{O}(M_{q}(n))$.
$\mathcal{A}$ has a standard grade defined by setting
$deg(x_{ij})=1$ for all $i$ and $j$.  The quantum determinant is a
central element of $\mathcal{A}$ given by
\begin{equation*}
det_{q}=\sum_{w \in S_{n}}(-q)^{l(w)}x_{{1}w({1})} \cdots x_{nw(n)}.
\end{equation*}
By adjoining the inverse of $det_{q}$ we obtain the quantum general
linear group
\begin{equation*}
\HH=\mathcal{O}(GL_{q}(n))=\mathcal{A}[det_{q}^{-1}].
\end{equation*}
$\HH$ is a Hopf algebra, and we denote the antipode of this algebra
by S. We will denote the element $x_{ij}$ by $u_{ij}$ when we are considering it as an element
of  $\HH$.
There is an adjoint coaction of $\HH$ on
$\mathcal{A}$, which, following \cite{DL}, we write as a right
coaction:
\begin{equation*}
\alpha_{q}:\mathcal{A} \rightarrow \mathcal{A} \otimes \HH.
\end{equation*}
given by
$$
\alpha_{q}(x_{ij})=\sum_{m,s=1}^{N}x_{ms}\otimes u_{sj}S(u_{im}).
$$
There is a variant of the adjoint coaction, denoted $\beta_{q}:\mathcal{A} \rightarrow \mathcal{A} \otimes \HH$, given by the formula
$$
\beta_{q}(x_{ij})=\sum_{m,s=1}^{N}x_{ms}\otimes S(u_{sj})u_{im}.
$$

\subsection{Invariants of the adjoint coaction}
An invariant of the coaction $\alpha_{q}$ is by definition an
element $b \in \mathcal{A}$ such that $\alpha_{q}(b)=b \otimes 1$.  (In \cite{DL} these are referred to as ``coinvariants''.) Let
$\I$ denote the set of invariants of $\A$ with respect to the coaction $\alpha_q$.  We let $\I'$ denote the set of invariants of $\A$ with respect to the coaction $\beta_{q}$. Notice that in the classical
($q=1$) case, the set $\I$ agrees with the usual invariants of the
action of $GL(n)$ on the coordinate ring of its Lie algebra,
$\mathcal{O}(M(n))$.

In \cite{DL}, Domokos and Lenagan explicitly determine $\I$.  Let us describe their result: for $1 \leq d \leq n$ and a subset $I=\{i_{1}< \cdots < i_{d}\} \subset \{1,...,n\},$ let
$det_{q,I}$ be the principal minor corresponding to $I$:
\begin{equation*}
det_{q,I}=\sum_{w \in S_{I}}(-q)^{l(w)}x_{i_{1}w(i_{1})} \cdots
x_{i_{d}w(i_{d})}
\end{equation*}
and set
\begin{equation*}
\Delta_{d}=\sum_{|I|=d}det_{q,I}.
\end{equation*}
Similarly set
\begin{equation*}
\Delta_{d}'=\sum_{|I|=d}q^{-2(\sum_{i\in I}i)}det_{q,I}.
\end{equation*}
It's not hard to see that $\Delta_{d} \in \I$ and $\Delta_{d}' \in \I'$ for every $d$
(\cite{DL}, Propositions 4.1 and 7.2).
\begin{theorem}(\cite{DL}, Corollary 6.2 and Theorem 7.3)
\label{dlthm} For $q\in \mathbb{C}^{\times}$ not a root of unity or $q=1$, $\I$ is a commutative polynomial algebra on the
$\Delta_{d}$, and similarly $\I'$ is a commutative polynomial algebra on the
$\Delta_{d}'$.
\end{theorem}

\section{Main results}
\subsection{Main Theorem}
We consider $\mathcal{A}$ as a left $\I$-module.  Our main result is
the following quantum analogue of Kostant's classical theorem:
\begin{theorem}
\label{mainthm}For $q\in \mathbb{C}^{\times}$ not a root of unity or $q=1$, $\mathcal{A}$ is a free graded left $\I$-module.
\end{theorem}

\begin{remark}
The condition on $q$ in the hypothesis of the theorem is needed only for the application of Theorem \ref{dlthm}.  It is conjectured that Theorem \ref{dlthm} holds for arbitrary $q$.
\end{remark}

\begin{remark}
The same result and proof hold for $\A$ regarded as a right $\I$-module, and, moreover, for $\A$ regarded as a left and right $\I'$-module.
\end{remark}

Before beginning the proof of this theorem we record a corollary.
Let $\I_{+}$
be the augmentation ideal of $\I$, i.e. $\I_{+}$ equals the elements
in $\I=\mathbb{C}[\Delta_{1},...,\Delta_{n}]$ with zero constant term.
Define $\I^{\mathcal{A}}$ to be the left ideal of $\mathcal{A}$
generated by $\I_{+}$.  By (\cite{DL}, Lemma 2.2) for $x \in
\mathcal{A}$ and $y \in \I$, $\alpha_{q}(xy)=\alpha_{q}(x)\alpha_{q}(y)$.
Therefore $\I^{\mathcal{A}}$ is an $\HH$-invariant graded subspace
of $\mathcal{A}$.

Set $\mathcal{H}=\A/\I^{\A}$.
Since $q$ is not a root of unity, the
representation theory of $\HH$ is semisimple (see e.g. [KS]) and so we can (non-canonically) identify
$\mathcal{A}$ with $\mathcal{H} \oplus \I^{\mathcal{A}}$ as graded $\HH$-comodules.
Now, by Theorem \ref{mainthm} and Lemma \ref{lemmaCG} we conclude:
\begin{corollary}
\label{maincor}
For $q\in \mathbb{C}^{\times}$ not a root of unity or $q=1$, the multiplication map in $\mathcal{A}$ gives an $\HH$-equivariant
isomorphism of graded vector spaces
\begin{equation*}
\mathcal{H} \otimes \I  \cong \mathcal{A}.
\end{equation*}
\end{corollary}

\subsection{Reflection Equation Algebras}

In this section we show that our main theorem has an analogue in the setting of reflection equation algebras.

The reflection equation algebra, denoted $\mathcal{S}$, is another quantization of the coordinate ring of $n \times n$ matrices due to Majid.  For a precise definition of $\mathcal{S}$ see \cite{D} and references therein.

For us, the most important properties of $\mathcal{S}$ are the following: $\mathcal{S}$ has an adjoint coaction of $\HH$, $\mathcal{S}$ is a comodule-algebra with respect to this action, and there exists a graded $\HH$-comodule isomorphism
$$
\Phi:\A \rightarrow \mathcal{S}
$$
intertwining the $\beta$-coaction on $\A$ with the adjoint coaction on $\mathcal{S}$.  The map $\Phi$ is not an algebra homomorphism.  Nevertheless, it does satisfy the property
\begin{equation}
\Phi(ab)=\Phi(a)\Phi(b)
\end{equation}
for $a \in \I'$ and $b \in \A$ (see the proof of Lemma 3.2 in
\cite{D}).

Let $\mathcal{J} \subset \mathcal{S}$ be the subalgebra of
invariants with respect to the adjoint coaction of $\HH$.  Since
$\Phi$ is a comodule isomorphism, $\mathcal{J}=\Phi(\I')$.  Since
$\mathcal{S}$ is a comodule-algebra, $\mathcal{J}$ is central.  Now
Theorem \ref{mainthm} implies the following.
\begin{theorem}
The algebra $\mathcal{S}$ is free as a $\mathcal{J}$-module.
\end{theorem}


We also have an analogue of Corollary \ref{maincor}.  Indeed, define
$\mathcal{J}^{\mathcal{S}}$ as we did $\I^{\A}$, and let
$\mathcal{H}'=\mathcal{S}/\mathcal{J}^{\mathcal{S}}$. In contrast to
$\mathcal{H}$, $\mathcal{H}'$ is an algebra which is a quantum
deformation of the coordinate ring of the nilpotent cone (see
\cite{D}).

\begin{corollary}
For $q\in \mathbb{C}^{\times}$ not a root of unity or $q=1$, we have a (non-canonical) $\HH$-equivariant
isomorphism of graded vector spaces
\begin{equation*}
\mathcal{H}' \otimes \mathcal{J}  \cong \mathcal{S}.
\end{equation*}
Note that this is an isomorphism of $\mathcal{J}$-modules, but the map is not an algebra morphism.
\end{corollary}

\section{The proof}
\subsection{Sketch of proof}
In this section we sketch the proof of Theorem \ref{mainthm}.  Our goal is to reduce the theorem to the following
standard fact:
\begin{proposition}
\label{fact} The polynomial algebra $\mathbb{C}[y_{1},...,y_{n}]$ in
$n$ indeterminates is a free module of rank $n!$ over the ring
symmetric polynomials $\mathbb{C}[y_{1},...,y_{n}]^{S_{n}}$.
Moreover the set
$$
\{y_{1}^{a_{1}} \cdots y_{n}^{a_{n}} : 0 \leq a_{i} < i \textit{ for
all } 1 \leq i \leq n \}
$$
is a basis.
\end{proposition}

We would like to mimic the proof in \cite{BL} and define a
filtration $F$ on $\A$ by setting $F^{d}\A=span\{x^{a}:trace(a) \leq
d \}$, and then appeal to Lemma \ref{BL}. The complication is that for $n \geq 3$ this filtration does
not preserve the algebra structure of $\A$. For example $F^{0} \A
\cdot F^{0} \A \nsubseteq F^{0} \A$ since for example
$$x_{23}x_{12}=x_{12}x_{23}-(q-q^{-1})x_{13}x_{22}.$$ To get around
this we will use a succession of filtrations, each one of which
slightly simplifies the quantum relations.

To explain the idea let us consider the case $\A=\mathcal{O}_{q}(M_{3}(\mathbb{C}))$.  Ignoring the relations of type (1)-(3), the
complicated (i.e. ``curly'') relations in $\A$ are
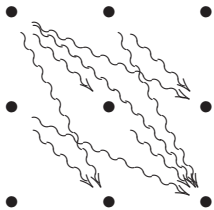
\begin{equation*}
\xymatrix{ {\bullet} \ar@{~>}[ddr] \ar@{~>}[dr] \ar@{~>}@/^/[ddrr]
\ar@{~>}[drr] & {\bullet} \ar@{~>}[ddr] \ar@{~>}[dr]
& {\bullet} \\
{\bullet} \ar@{~>}[dr] \ar@{~>}[drr] & {\bullet} \ar@{~>}[dr] & {\bullet} \\
{\bullet} & {\bullet} & {\bullet} }
\end{equation*}

Let $F$ be the filtration on $\A$
defined by
\begin{equation*}
F^{d} \A = span \{ x^{a} : \sum_{|i-j|< 2} a_{ij} \leq d \}.
\end{equation*}
$F$ preserves the algebra structure of $\A$ (cf. Lemma \ref{St} below) and so we consider the associated graded algebra
$\A'=gr_{F}\A$.  In $\A'$ most of the complicated relations disappear and we are
left with
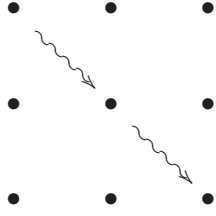
\begin{equation*}
\xymatrix{ {\bullet} \ar@{~>}[dr] & {\bullet}
& {\bullet} \\
{\bullet} & {\bullet} \ar@{~>}[dr] & {\bullet} \\
{\bullet} & {\bullet} & {\bullet} }
\end{equation*}

Our next step is to introduce a filtration on $\A'$ which will
further simplify the relations.  $\A'$ has a PBW type basis (cf. Lemma \ref{nonsenselemma}(3) and Lemma \ref{St}(4) below), which
by abuse of notation, we can continue to denote as $\{ x^{a}\}$. Define a filtration
$F'$ on $\A'$ by
\begin{equation*}
F'_{d} \A' = span \{ x^a : \sum_{|i-j|<1} a_{ij} \leq d \}.
\end{equation*}

We show below that $F'$ is compatible with the product in $\A'$, and
hence we can consider $\A''=gr_{F'}\A'$.  Now, in $\A''$ all the
complicated relations disappear.  Moreover, the image, $\I''$, of
the subalgebra $\I$ in $\A''$ consists of the symmetric polynomials
in the diagonal entries.  Therefore it is easy to see that $\A''$ is
free over $\I''$ using Proposition \ref{fact}.  By Lemma \ref{BL} we
conclude our result.
\subsection{$q$-Mutation Systems}

We now introduce the terminology needed to handle
successions of filtrations.

Suppose we have an ordered set $\{I,\leq\}$ and a
collection of indeterminates $\{x_{i}\}_{i \in I}$.  We would
like to discuss an algebra on the $\{x_{i}\}$ subject to certain
commutation relations.  Let $\mathcal{F}$ be the free algebra on
the $\{x_{i}\}$.  A \textbf{standard monomial} $x_{i_{1}} \cdots x_{i_{l}}
\in \mathcal{F}$ is one such that $i_{1} \leq \cdots \leq i_{l}$.

\begin{definition}
$ $
\begin{enumerate}
\item
A $\mathbf{q}$-\textbf{mutation system} is a tuple $S=(\{q_{ij}\},\{f_{ij}\})$ where
$i < j \in I$, $q_{ij} \in \mathbb{C}^{\times}$, and $f_{ij} \in
\mathcal{F}$. We denote by $\mathcal{A}(S)$ the quotient of $\mathcal{F}$ by
the two-sided ideal generated by
$x_{j}x_{i}-(q_{ij}x_{i}x_{j}+f_{ij})$.
\item
Let  $\xi=x_{i_{1}} \cdots x_{i_{l}}$ be a monomial in
$\mathcal{F}$, and suppose there exists $r$ such that $i_{r} <
i_{r-1}$.  Then an \textbf{elementary mutation} of $\xi$ in the $r^{th}$
position is the sum of elements
$$
(q_{i_{r}i_{r-1}}x_{i_{1}}\cdots
x_{i_{r-2}}x_{i_{r}}x_{i_{r-1}}x_{i_{r+1}}\cdots x_{i_{l}} )+(
x_{i_{1}}\cdots x_{i_{r-2}}f_{i_{r}i_{r-1}}x_{i_{r+1}} \cdots
x_{i_{l}}).$$ A elementary mutation of a polynomial $f \in
\mathcal{F}$ is the polynomial obtained by an elementary mutation of
one of its monomials.
\item A $q$-mutation system $S$ has \textbf{finite
mutation property} (FMP) if any monomial $x_{i_{1}} \cdots x_{i_{l}}$
can be transformed into a linear combination of standard monomials
using finitely many elementary mutations.
\item The $q$-mutation system $S$ satisfies \textbf{Poincare-Birkhoff-Witt property} (PBW) if the images of standard
monomials form a basis in $\mathcal{A}(S)$.
\end{enumerate}

\end{definition}

%
%
A \textbf{weighting} of $I$ is a function $w:I \rightarrow
\mathbb{Z}_{\geq0}$.  A weighting $w$ defines a filtration
$F_{w}$ of $\mathcal{F}$ by $deg_{F_{w}}x_{i_{1}} \cdots
x_{i_{l}}=\sum w(i_{k})$.  If a q-mutation system $S$
satisfies the PBW property then a weighting $w$ defines a linear
filtration $F_{w,S}$ on $\mathcal{A}(S)$ in a natural way. Precisely,
$F_{w,S}^{d}\mathcal{A}(S)$ is the span of all images of standard monomials
$\xi$ such that $deg_{F_{w}}\xi \leq d$.

We call a weighting $w$
\textbf{compatible} with $S=(\{q_{ij}\},\{f_{ij}\})$ if
for all $i < j$, $deg_{F_{w}}f_{ij} \leq w(i)+w(j)$.  If a weighting $w$ is compatible with
$S=(\{q_{ij}\},\{f_{ij}\})$ then we define a $q$-mutation system
$$\sigma_{w}(S)=(\{q_{ij}\},\{\sigma_{F_{w}}^{w(i)+w(j)}(f_{ij})\}).$$  Here we identify the linear spaces $gr_{F_{w}}\mathcal{F}$ with $\mathcal{F}$.

\begin{lemma}
\label{nonsenselemma}
 Let $S=(\{q_{ij}\},\{f_{ij}\})$ be a $q$-mutation system
with a compatible weighting $w$.  Suppose $S$ satisfies the FMP and
PBW properties.  Consider the natural projection $p:\mathcal{F}
\rightarrow \mathcal{A}(S)$. Then,
\begin{enumerate}
\item  $p(F_{w}^{d}\mathcal{F}) =
F_{w,S}^{d}\mathcal{A}(S)$.
\item The linear filtration $F_{w,S}$ is
compatible with the algebra structure of $\mathcal{A}(S)$.
\item Suppose $\sigma_{w}(S)$ satisfies FMP.  Then there is a natural isomorphism
$ gr_{F_{w,S}}\mathcal{A}(S) \cong \A(\sigma_{w}(S)). $
\item $\sigma_{w}(S)$ satisfies the PBW property.
\end{enumerate}
\end{lemma}
\begin{proof}
To prove (1) note that by definition of $F_{w,S}$ we have the inclusion $p(F_{w}^{d}\mathcal{F}) \supset F_{w,S}^{d}\mathcal{A}(S)$.
Conversely, let $\xi \in F_{w}^{d}\mathcal{F}$.  Since $S$ satisfies FMP there exist elements
$\xi_{1}=\xi, \xi_{2},...,\xi_{n} \in \mathcal{F}$ such that
$\xi_{n}$ is a linear combination of standard monomials, and
$\xi_{i+1}$ is an elementary mutation of $\xi_{i}$ for all $i$. By
the compatibility condition,
$$
deg_{F_{w}}(\xi_{1}) \geq deg_{F_{w}}(\xi_{2}) \geq \cdots \geq
deg_{F_{w}}(\xi_{n}).
$$
Therefore $\xi_{n} \in F_{w}^{d}\mathcal{F}$, and since it's a
linear combination of standard monomials $p(\xi_{n}) \in
F_{w,S}^{d}\mathcal{A}(S)$.  Since moreover $p(\xi_{1})= \cdots = p(\xi_{n})$, we
conclude $p(\xi) \in F_{w,S}^{d}\mathcal{A}(S)$.

Part (2) follows from part (1).

For part (3), the natural morphism $f:\mathcal{A}(\sigma_{w}(S)) \rightarrow
gr_{F_{w,S}}\mathcal{A}(S)$ is surjective.  Now note that the weighting $w$
defines a grading on $\mathcal{F}$.  Since the relations defining
$\mathcal{A}(\sigma_{w}(S))$ are homogenous, $\mathcal{A}(\sigma_{w}(S))$ inherits an
induced grading, the $d^{th}$ component of which we denote
$\mathcal{A}^{d}(\sigma_{w}(S))$. The morphism $f$ is clearly graded, and
hence $f(\mathcal{A}^{d}(\sigma_{w}(S)))=gr_{F_{w,S}}^{d}\mathcal{A}(S)$, which implies
$$
dim(\mathcal{A}^{d}(\sigma_{w}(S)) \geq dim(gr_{F_{w,S}}^{d}\mathcal{A}(S)).
$$
To see that $ dim(\mathcal{A}^{d}(\sigma_{w}(S)) \leq
dim(gr_{F_{w,S}}^{d}\mathcal{A}(S))$ note that since $\sigma_{w}(S)$ satisfies
FMP, the images of standard monomials of degree $d$ span
$\mathcal{A}^{d}(\sigma_{w}(S))$. Since $S$ satisfies the PBW property, the
standard monomials of degree $d$ form a basis for
$gr_{F_{w,S}}^{d}\mathcal{A}(S)$.

Finally, part (4) follows from (3).
\end{proof}

\subsection{Proof of main theorem}

We now specialize the terminology introduce above to our case.  Let $I=\{ (i,j) :
i,j=1,...,n \}$ be ordered lexicographically, i.e. $(i,j) \leq (k,l)$ if, and only if, $ni+j \leq nk+l$.  We introduce
a family $\{S_{t}\}$ of $q$-mutation systems for $t=1,...,n$.

Let us first define the $q$-mutation system that's naturally associated
to $\A$.  Let $S_{1} = (\{q_{ij,kl}\},\{f_{ij,kl}\})$ where
%
%
$$
q_{ij,kl} = \left\{
              \begin{array}{ll}
                1 & \hbox{if \xymatrix{ {x_{ij}} \ar@{-}[r] & {x_{kl}}} or \xymatrix{ {x_{ij}} \ar@{~>}[r] & {x_{kl}}};} \\
                q^{-1} & \hbox{if \xymatrix{ {x_{ij}} \ar[r] & {x_{kl}}}.}
              \end{array}
            \right.
$$
and
$$
f_{ij,kl} = \left\{
              \begin{array}{ll}
                0 & \hbox{if \xymatrix{ {x_{ij}} \ar@{-}[r] & {x_{kl}}} or \xymatrix{ {x_{ij}} \ar[r] & {x_{kl}}};} \\
                (q^{-1}-q)x_{il}x_{kj} & \hbox{if \xymatrix{ {x_{ij}} \ar@{~>}[r] & {x_{kl}}}.}
              \end{array}
            \right.
$$
It's clear that $\A(S_{1})=\A$.

Now let $t \in  \{1,..., n\}$.  For $\iota = (i,j) \in I$, let $\epsilon(\iota)=|i-j|$.  Let $w_{t}$ be the weighting defined by $w_{t}(\iota)=1$ if
$\epsilon(\iota) < n-t$ and zero else.  Define $S_{t}= (\{q_{ij,kl}\},\{f_{ij,kl}^{(t)}\})$ to be the $q$-mutation system where the scalars $q_{ij,kl}$ are the same as above and,
$$
f_{ij,kl}^{(t)} = \left\{
              \begin{array}{ll}
                0 & \hbox{if \xymatrix{ {x_{ij}} \ar@{-}[r] & {x_{kl}}} or \xymatrix{ {x_{ij}} \ar[r] & {x_{kl}}};} \\
                (q^{-1}-q)x_{il}x_{kj}w_{t-1}(i,l)w_{t-1}(k,j) & \hbox{if \xymatrix{ {x_{ij}} \ar@{~>}[r] & {x_{kl}}}.}
              \end{array}
            \right.
$$
Set $\A_{t}=\A(S_{t})$ to be the algebra associated to $S_{t}$.

\pagebreak

\begin{lemma}
\label{St}
$ $
\begin{enumerate}
\item The weighting $w_{t}$ is compatible with $S_{t}$.
\item $S_{t}$ and $S_{t+1}$ are related by
$\sigma_{w_{t}}(S_{t})=S_{t+1}$.
\item $S_{t}$ satisfies that FMP property.
\item $S_{t}$ satisfies that PBW property.
\end{enumerate}
\end{lemma}

\begin{proof}

For part (1) suppose $\xymatrix{ {x_{ij}} \ar@{~>}[r] & {x_{kl}}}$.  Note that then
$$
max\{\epsilon(il),\epsilon(kj)\} > max\{\epsilon(ij),\epsilon(kl)\}.
$$
We want to show that $deg_{F_{w_{t}}}f_{ij,kl}^{(t)} \leq
deg_{F_{w_{t}}}x_{ij}x_{kl}$.  The only nontrivial case is when
$w_{t-1}(il)=w_{t-1}(kj)=1$.  In this case
$max\{\epsilon(il),\epsilon(kj)\} \leq n-t$, and so
$max\{\epsilon(ij),\epsilon(kl)\} < n-t$ and so
$w_{t}(ij)=w_{t}(kl)=1$.  Then
$$
deg_{F_{w_{t}}}x_{ij}x_{kl}=2 \geq deg_{F_{w_{t}}}f_{ij,kl}^{(t)}.
$$

To show part (2) we must prove that
$$
f_{ij,kl}^{(t+1)}=\sigma_{F_{w_{t}}}^{w(ij)+w(kl)}(f_{ij,kl}^{(t)})
$$
If $f_{ij,kl}^{(t)}=0$ then $f_{ij,kl}^{(t+1)}=0$, so the only
nontrivial case is when $f_{ij,kl}^{(t)} \neq 0$.  As in the previous
case, this only happens when $w_{t}(ij)+w_{t}(kl)=2$.  Therefore we
have to show that $deg_{F_{w_{t}}}f_{ij,kl}^{(t)}=2$ if, and only
if, $w_{t}(il)w_{t}(kj)=1$.  But this is clear since
$deg_{F_{w_{t}}}f_{ij,kl}^{(t)}=w_{t}(il)+w_{t}(kj)$.

For part (3) we define the descent statistic on an element of $\mathcal{F}$
by first defining
$$
des(x_{\iota_{1}} \cdots x_{\iota_{n}})=\sharp\{(k,l): k<l \text{
and } \iota_{k} > \iota_{l} \}.
$$
Extend this definition to an arbitrary element in $\mathcal{F}$ by
$$des(\sum \xi_{i})=max\{des(\xi_{i})\},$$ where $\xi_{i}$ are
monomials in $\mathcal{F}$.  To prove that $S_{t}$ satisfies FMP it
clearly suffices to show that if $\xi'$ is an elementary mutation of
$\xi$ then
$$
des(\xi')<des(\xi).
$$
For this it is enough to show that if $\iota_{k} < \iota_{k-1}$,
then
$$
des(x_{\iota_{1}} \cdots x_{\iota_{n}})>des(x_{\iota_{1}} \cdots
x_{\iota_{k-1}}f^{(t)}_{\iota_{k}\iota_{k-1}}x_{\iota_{k+2}}\cdots
x_{\iota_{n}}).
$$

The only nontrivial case is when $\xymatrix{ x_{\iota_{k}} \ar@{~>}[r] &
x_{\iota_{k-1}}}$.  This is immediate from our definition of
$f^{(t)}_{\iota_{k}\iota_{k-1}}$ and the definition of lexicographic
ordering.

To prove part (4) first note that $S_{1}$ satisfies the PBW property
by Theorem \ref{PBW}.  Now, by induction, Lemma \ref{nonsenselemma},
and part (2) above, we conclude that $S_{t}$ satisfies the PBW
property.

\end{proof}

By Lemma \ref{nonsenselemma}(2) and Lemma \ref{St}(2), we make the
identification
$$
gr_{F_{w_{t}}}\A_{t}=\A_{t+1}.
$$
Notice that the algebra $\A_{n}$ has no complicated ``curly''
relations.  We now want to use Lemma \ref{BL} to reduce Theorem
\ref{mainthm} to Proposition \ref{fact}. In order to do this we
first consider the behavior of the algebra $\I$ with respect to the
succession of filtrations $F_{w_{t}}$.

\begin{definition}
Let $\I_{1}=\I$ and define $\I_{t} \subset \A_{t}$ by induction to
be the associated graded algebra $gr_{F_{w_{t-1}}}\I_{t-1}$, where
$\I_{t-1} \subset \A_{t-1}$ inherits the induced filtration from
$F_{w_{t-1}}$.
\end{definition}

\begin{proposition}
The algebra $\I_{t}$ is generated by $\{\Delta_{d}^{(t)}:t=1,..., n \}$ where
$$
\Delta_{d}^{(t)} = \sum_{I=\{i_{1}< \cdots < i_{d}\}} \sum_{\substack{ w \in S_{I} \\ |i-w(i)| \leq
n-t }} (-q)^{l(w)}x_{i_{1}w(i_{1})} \cdots x_{i_{d}w(i_{d})}.
$$
\end{proposition}

\begin{proof}
Define $\mathcal{O} \subset \mathcal{F}$ to be the two-sided ideal generated by $\{ x_{ij} : i \neq j \}$.
Note that $\mathcal{O}$ is invariant under mutations with respect to any system $S_{t}$.  Let $y_{1},...,y_{n}$ be indeterminates
and set $\F_{n}$ to be the free algebra on the $\{ y_{i} \}$.

Given an element $h \in \F_{n}$ we can consider the evaluation
$h(x_{11},...,x_{nn}) \in \F$.  Now, for any $t \in \{1,...,n\}$ we can preform a sequence of (finitely many) elementary mutations on
$h(x_{11},...,x_{nn})$ (with respect to $S_{t}$) to obtain an element of the form $h'(x_{11},...,x_{nn})+f$.  Here $h'(x_{11},...,x_{nn})$ is a linear combination of standard monomials and $f \in \mathcal{O}$.  Note also that
$deg_{F_{w_{t}}}h(x_{11},...,x_{nn})=deg_{F_{w_{t}}}h'(x_{11},...,x_{nn})$, and $$deg_{F_{w_{t}}}f \leq deg_{F_{w_{t}}}h(x_{11},...,x_{nn}).$$

Let $p:\F \rightarrow \A_{t}$ be the natural projection.  It follows from the previous assertion that for elements $h,f$ as above,
\begin{equation}
\label{equality}
deg_{F_{w_{t}},S_{t}}(p(h(x_{11},...,x_{nn})+f))=deg_{F_{w_{t}}}(h(x_{11},...,x_{nn})).
\end{equation}
Indeed, if $h(x_{11},...,x_{nn})$ and $f$ are both combinations of standard monomials then this is obvious.  If only $h(x_{11},...,x_{nn})$ is a combination of standard monomials then we can apply mutations to $f$ to reduce to the previous case since the mutations can only decrease the
degree of $f$ and leave $f \in \mathcal{O}$.  Finally, if neither are a combination of standard monomials then we can apply mutations to
$h(x_{11},...,x_{nn})$ to reduce to the previous case.

Define a weighting $u$ of $\{1,...,n\}$ by $u(i)=i$.  Then by (\ref{equality}), for $h \in \F_{n}$,
$$deg_{F_{w_{t},S_{t}}}(h(\Delta_{1}^{t},...,\Delta_{n}^{t}))=deg_{F_{u}}(h).$$  Lemma \ref{symbprod} implies that
\begin{eqnarray*}
\sigma_{F_{w_{t},S_{t}}}(h(\Delta_{1}^{t},...,\Delta_{n}^{t}))&=&
\sigma_{F_{u}}(h)(\sigma_{F_{w_{t}}}(\Delta_{1}^{(t)}),...,\sigma_{F_{w_{t}}}(\Delta_{n}^{(t)}))\\&=&
\sigma_{F_{u}}(h)(\Delta_{1}^{(t+1)},...,\Delta_{n}^{(t+1)})
\end{eqnarray*}
By induction on $t$ this implies the assertion.

%
%
\end{proof}

We now have all the ingredients to prove Theorem \ref{mainthm}.
Indeed, by Proposition \ref{fact} the set of standard monomials
$$
\{\prod_{\iota \in I}x_{\iota}^{r_{\iota}} : r_{(i,i)} \leq i \text{
for } 1 \leq i \leq n \}
$$
is a free basis of $\A_{n}$ over $\I_{n}$.  Therefore repeated
application of Lemma \ref{BL} shows that these monomials form a free
basis of $\A$ over $\I$.  This completes the proof of the main theorem.

\end{document}